\documentclass[12pt]{article}

\usepackage{latexsym,amsfonts,amsmath,
amssymb,epsfig,tabularx,amsthm,dsfont,enumerate}

\makeatletter
\newcommand{\LeftEqNo}{\let\veqno\@@leqno}
\makeatother

\theoremstyle{definition}
\newtheorem{thm}{Theorem}
\newtheorem{rem}[thm]{Remark}
\newtheorem{lem}[thm]{Lemma}
\newtheorem{prop}[thm]{Proposition}
\newtheorem{cor}[thm]{Corollary}

\newcommand\R{{\mathbb{R}}}

\newcommand\N{\mathbb{N}}
\newcommand\Z{\mathbb{Z}}

\newcommand\eps{\varepsilon}
\renewcommand\a{\alpha}

\newcommand\eu{{\rm e}}

\newcommand{\e}{\varepsilon}
\newcommand{\wt}{\widetilde}
\newcommand{\widebar}[1]{\mbox{\kern1pt\hbox{
\vbox{\hrule height 0.5pt \kern0.25ex
        \hbox{\kern-0.05em \ensuremath{#1 }\kern-0.05em}}}}\kern-0.1pt}
\newcommand\il{\left<}
\newcommand\ir{\right>}

\newcommand{\dx}{\, {\rm d} x }

\newcommand{\abs}[1]{\left\vert #1 \right\vert}
\newcommand{\norm}[1]{\left\Vert #1 \right\Vert}

\newcommand{\F}{\mathcal{F}}

\newcommand{\A}{\mathcal{A}}
\newcommand{\prob}{{\bf p}}

\newlength{\fixboxwidth}
\renewcommand{\rho}{{\varrho }}

\title{Complexity of Oscillatory Integrals on the Real Line}

\author{Erich Novak $^{a}$,\ \
Mario Ullrich $^{b}$,\ \ Henryk Wo\'zniakowski $^{c, d}$,\ \ Shun Zhang $^{e,}$\footnote{Corresponding author.
\newline\indent\, Email addresses: erich.novak@uni-jena.de(E. Novak), mario.ullrich@jku.at (M. Ullrich), henryk@cs.columbia.edu (H. Wo\'zniakowski), shzhang27@163.com (S. Zhang).}\ \
\\ {\footnotesize $^{\rm a}$ \rm Mathematisches Institut, Universit\"at Jena,
Ernst-Abbe-Platz 2, 07743 Jena, Germany}
\\ {\footnotesize\rm $^{\rm b}$ Institut f\"ur Analysis, Johannes Kepler Universit\"at
Linz, Austria}
\\ {\footnotesize $^{\rm c}$ \rm Department of Computer Science, Columbia University,
New York, NY 10027, USA}
\\  {\footnotesize\rm $^{\rm d}$ Institute of Applied Mathematics, University of Warsaw,
ul. Banacha 2, 02-097 Warszawa, Poland}
 \\  {\footnotesize\rm $^{\rm e}$ School of Computer Science and Technology, Anhui
University,
 Hefei 230601, China}}
\begin{document}

\maketitle

\begin{abstract}
We analyze univariate oscillatory integrals defined on the real line
for functions from the standard Sobolev space $H^s (\R)$
and from the space $C^s(\R)$
with an arbitrary integer $s\ge1$.
We find tight upper and lower bounds for the worst case
error of optimal algorithms that use~$n$ function values.
More specifically, we study integrals of the form
\begin{equation}\label{eq:problem}
I_k^\rho (f) = \int_{\R} f(x) \,\eu^{-i\,kx} \rho(x)  \, {\rm d} x\ \
\
\mbox{for}\ \ f\in H^s(\R)\ \ \mbox{or}\ \ f\in C^s(\R)
\end{equation}
with $k\in\R$ and a smooth density function $\rho$ such as
$ \rho(x) =  \frac{1}{\sqrt{2 \pi}} \exp( -x^2/2) $.
The optimal error bounds are
$\Theta((n+\max(1,|k|))^{-s})$ with the factors in the $\Theta$ notation
dependent only on $s$ and $\rho$.

\end{abstract}
\noindent{\bf Key words:}\, Oscillatory integrals;
Complexity;
Sobolev space.
\\
{\bf AMS Classification:}\,
 65Y20,~\,42B20,~\,65D30,~\,68Q25.

\section{Introduction}

In the last decades, many papers have been published on
the approximate computation of
highly oscillatory univariate integrals over finite intervals,
see the two surveys of Huybrechs and Olver \cite{HO09},
Milovanovi\'c and Stani\'c \cite{MiSt} and papers cited there.
Our paper \cite{NUW13}
belongs to this group of papers. We studied the
integration interval $[0,1]$ and found tight lower and upper
error bounds for algorithms that use $n$ function values
for periodic and nonperiodic
functions from the standard Sobolev spaces
$H^s([0,1])$ with an integer $s\ge1$.

For the case when the integration interval is
unbounded, the literature is not so rich.
We refer the readers to Blakemore, Evans and Hyslop \cite{BEH76},
Chen \cite{Ch14} and  Xu, Milovanovi\'c and Xiang \cite{XMX15}
for pointers to the literature.
However,  we could not find any
paper where tight lower error bounds were found.

The aim of this paper is to generalize results of \cite{NUW13}
for oscillatory integrals of the form~\eqref{eq:problem}
defined over the real line for functions from the space
$H^s(\R)$ with smooth density functions such as the normal one.
The main, and possibly surprising, result of this paper is that
for the real line and the space $H^s(\R)$,
sharp error bounds for algorithms that use $n$ function values are
roughly the same as for the interval $[0,1]$ and the periodic space
$H^s([0,1])$. More precisely, they are of order
$(n+\max(1,|k|))^{-s}$.
We add in passing that sharp error bounds are
higher for the density
$\rho = 1_{[0,1]}$ if we consider the whole class $H^s([0,1])$
without additional conditions on the boundary.

To approximate the univariate oscillatory integrals~\eqref{eq:problem},
we use a smooth partition of unity, and reduce
the integration problem over the whole real line
to the case of the integration problem over finite intervals.
The last problem could be solved by the change of variables and the use
of the results from~\cite{NUW13} for the integration domain $[0,1]$.
However this approach has some
drawbacks. First of all, we assume in \cite{NUW13} that $k$ is an
integer, which is not required in this paper. We also used
a slightly different norm of the space $H^s([a,b])$ than the more
standard norm which is now used.
Furthermore, the change of variables yields to larger factors
in the upper error bounds. Finally, and more importantly,
we present a new proof technique
which is based on Poisson's summation formula
as the basic tool to obtain upper error bounds.
That is why we decided to use this new approach and not to use
the results from \cite{NUW13}.

Sharp error bounds allow us to find sharp estimates on the information
complexity which is defined as the minimal number of function values
needed to find an algorithm with an error $\e\cdot{\rm CRI}$.
Here, $\e$ is a presumably small error threshold and ${\rm CRI}=1$
when the absolute error criterion is used, and ${\rm CRI}$ is the
initial error obtained by the zero algorithm when the normalized error
criterion is used.

Consider first the absolute error criterion.
The information complexity is then  roughly
$c_{s,\rho}\,\e^{-1/s}-\max(1,|k|)$ for some positive $c_{s,\rho}$.
Hence, large $|k|$ helps for not too small $\e$ and is irrelevant
if $\e$ goes to zero.

Consider now the normalized error
criterion. Then the information complexity is roughly
$\max(1,|k|)\,(c_{s,\rho}\,\e^{-1/s}-1)$ again for some positive
$c_{s,\rho}$. In this case, the information complexity is proportional
to $\max(1,|k|)$ so that large $|k|$ hurts for all $\e<1$.

The paper is organized as follows. In Section 2 some definitions and
preliminaries are given. In Section 3 we
study integration of functions with compact support,
whereas in Section 4 we consider integration of
functions defined over the real line. In both cases, we find matching
lower and upper error bounds for algorithms that use $n$ function values.

\section{Preliminaries}\label{sec:prelim}

In this paper we consider real or complex valued functions defined on
the whole real~line $\Omega = \R$ or on an interval $\Omega = [a,b]$
with $-\infty<a<b<\infty$.
Let $\il\cdot,\cdot\ir_{0,\Omega}$ be the usual inner product in $L_2(\Omega)$,
i.e.~$\il f,g\ir_{0,\Omega}:=\int_\Omega f(x) \widebar{g(x)}\dx$.
We consider the standard Sobolev space
$$
H^s(\Omega)=\{ f\in L_2(\Omega) \mid
f^{(s-1)} \hbox{ is abs. cont., } f^{(\ell)} \in L_2(\Omega)
\ \ \mbox{for}\ \ \ell=0,1,\dots,s \}\qquad
\mbox{for}\ \ s\in\N,
$$
which is equipped with the inner product
$$
\il f,g \ir_{s,\Omega} \;=\; \sum_{\ell=0}^{s} \il f^{(\ell)},
g^{(\ell)}\ir_{0,\Omega}\ \ \ \ \mbox{for}\ \ f,g\in H^s(\Omega).
$$
The norm in $H^s(\Omega)$ is given by
$\|f\|_{H^s(\Omega)}=\il f,f \ir_{s,\Omega}^{1/2}$.

We also consider the space $C^s(\Omega)$ of $s$ times
continuously differentiable functions with the norm
\[
\|f\|_{C^s(\Omega)} \,:=\, \max_{\ell=0,1,\dots,s}\,
\sup_{x\in\Omega}\, \bigl|f^{(\ell)}(x)\bigr|.
\]

Moreover, we consider functions with compact support
and define the respective classes
$H^s_0(\Omega)$ and $C^s_0 (\Omega)$.
More exactly, for functions $f\in H^s_0([a,b])$
we assume that $f\in H^s([a,b])$ and
$$
f^{(\ell)}(a)=f^{(\ell)}(b)=0\ \ \ \mbox{for}\ \ \ell=0,1,\dots,s-1,
$$
and for functions $f \in C^s_0 ([a,b])$ we assume that
that $f\in C^s([a,b])$ and
$$
f^{(\ell)}(a)=f^{(\ell)}(b)=0\ \ \ \mbox{for}\ \ \ell=0,1,\dots,s.
$$

Given a nonzero and
non-negative integrable function $\rho\colon\Omega\to[0,\infty)$,
we consider the approximation of oscillatory integrals of the
form~\eqref{eq:problem}, i.e.,
$$
I_k^\rho (f) = \int_{\Omega} f(x) \,\eu^{-i\,kx} \rho(x)  \, {\rm d} x,
\ \ \ i=\sqrt{-1},
$$
where $k \in \R$ and $f\in F$,
where $F\in\{H^s(\Omega),H^s_0(\Omega),C^s(\Omega),C^s_0(\Omega)\}$.
Specific smoothness assumptions on $\rho$ are given in the
corresponding theorems. For $\Omega=\R$, these assumptions are
satisfied for the normal density
\[
\rho(x) =  \frac{1}{\sqrt{2 \pi}} \exp( -x^2/2)\ \ \ \mbox{for}
\ \ x\in \R,
\]
whereas for $\Omega=[a,b]$, we study $\rho = 1_{[a,b]}$
which was already considered in \cite{NUW13} for $[a,b]=[0,1]$ and
an integer $k$.

For the approximation of $I_k^\rho$ we consider
algorithms that use $n$ function values.
It is well known that linear algorithms $A_n$ are optimal in our setting,
see e.g. \cite{TWW88} or \cite{NW08}, hence there is no need to study
more general algorithms such as nonlinear or adaptive algorithms.
The linear algorithms, or quadrature formulas, are
of the form
\[
A_n(f) = \sum_{j=1}^n a_j f(x_j) ,
\]
where the coefficients $a_j$ and the nodes $x_j$
of course may depend on $\Omega$, $s$, $k$, $\rho$ and $n$.

The aim of this paper is to prove upper and lower bounds on the
$n$th minimal (worst case) errors
\[
e(n, I_k^\rho , F) :=
\inf_{A_n}\, \sup_{f\in F:\, \|f\|_F\le1}\, |I_k^\rho (f) - A_n(f)|.
\]
This number is
the worst case error on the unit ball of $F$
of an optimal algorithm $A_n$ that uses at most $n$ function values
for the approximation of the functional $I_k^\rho$.
The initial error is given for $n=0$ when we do not sample the
functions. In this case the best we can do is to take
the zero algorithm $A_0(f)=0$ and
\[
e(0, I_k^\rho , F) :=
\sup_{f\in F:\, \|f\|_F\le1}\, |I_k^\rho (f)| = \| I_k^\rho \|_{F}
\]
We are ready to define
the \emph{information complexity}, which is the minimal number $n$ of
function values for which the $n$th minimal error of at most
$\eps\,{\rm CRI}$. Here, ${\rm CRI}=1$ if we consider the absolute
error criterion, and ${\rm CRI}=e(0,I_k^\rho,F)$ if we consider the
normalized error criterion. Hence,
for the \emph{absolute error criterion}
the information complexity is defined as
\begin{equation}\label{eq:absolute}
n^{\rm abs}(\eps, I_k^\rho , F) :=
\min\bigl\{n\colon e(n, I_k^\rho , F) \le \eps \bigr\},
\end{equation}
while for the \emph{normalized error criterion}
the information complexity is defined as
\begin{equation}\label{eq:normalized}
n^{\rm nor}(\eps, I_k^\rho, F) :=
\min\bigl\{n\colon e(n, I_k^\rho, F) \le \eps\; e(0, I_k^\rho, F) \bigr\}.
\end{equation}

As already mentioned, our basic tool to derive upper error bounds
will be the Poisson summation formula. We now remind the reader of this
formula. For integrable functions~$f$ on the whole real line,
the Fourier transform of $f$ is defined by
\[
\left[ \F f \right] (z) \;=\; \int_\R f(y)\, \eu^{-2\pi i zy} {\rm d} y
\qquad \mbox{for}\ \ z\in\R.
\]
The study of quadrature rules with equidistant nodes
can be done by Poisson's summation formula, see e.g.~\cite[Thm.~VII.2.4]{SW71}.
We state here only the univariate
version.
\begin{lem}\label{lem:poisson}
Let $f\in L_1(\R)$ be continuous. Then its periodization
$$
g(x):=\sum_{m\in\Z}f(x+m)\ \ \ \mbox{for}\ \ x\in\R
$$
converges in
the norm of $L_1([0,1])$. The resulting (1-periodic) function has the Fourier
expansion
\[
g(x)=\sum_{z\in\Z} \left[  \F f \right] (z)\,\eu^{2\pi i zx}\ \ \
\mbox{for}\ \ x\in\R.
\]
\qed
\end{lem}

A consequence of Lemma~\ref{lem:poisson}
applied to the function $cf(\cdot)\eu^{-i k\cdot}$
for an integrable and continuous $f$,
real $c\not=0$ and $k\in \R$, and then taking $g(0)$,
yields
\begin{equation}\label{eq:poisson}
c \sum_{x\in c\Z} f(x)\,\eu^{-i k x} \;=\; \sum_{z\in\Z}
\left[  \F f \right] \Bigl(\frac{z}c+\frac{k}{2\pi}\Bigr),
\end{equation}
see e.g.~\cite[Lemma~12]{MU14}.

Furthermore, we need the following lemma.

\begin{lem}\label{lem:norm}
Let $s\ge1$. For every $\Omega\subset\R$ we have
\begin{itemize}
\item (i)
\[
\|f g\|_{H^s(\Omega)} \;\le\; 2^{s} \, \|f\|_{H^s(\Omega)}\, \|g\|_{C^s(\Omega)}
\ \ \ \mbox{for} \ \ f\in H^s(\Omega) \ \mbox{and}\ g\in C^s(\Omega),
\]
\item (ii)
$$
\|f g\|_{C^s(\Omega)}\;\le\,2^{s}\,\|f\|_{C^s(\Omega)}\,
\|g\|_{C^s(\Omega)}
\ \ \ \mbox{for}\ \ f\in C^s(\Omega) \ \mbox{and}\ g\in C^s(\Omega).
$$
\end{itemize}
\qed
\end{lem}

\begin{proof}
(i) Using the product rule and the Cauchy-Schwarz inequality we obtain
\[\begin{split}
\|f g\|_{H^s(\Omega)}^2 \,&=\, \sum_{\ell=0}^s \norm{(f g)^{(\ell)}}_{L_2(\Omega)}^2
\,\le\, \sum_{\ell=0}^s \left(\sum_{m=0}^\ell \binom{\ell}{m} \norm{f^{(m)}
g^{(\ell-m)}}_{L_2(\Omega)}\right)^2 \\
&\le\, \sum_{\ell=0}^s 2^\ell\,\sum_{m=0}^\ell
\binom{\ell}{m} \norm{f^{(m)} g^{(\ell-m)}}_{L_2(\Omega)}^2
\,\le\, 2^s\,\norm{g}_{C^s(\Omega)}^2\,\sum_{\ell=0}^s \sum_{m=0}^\ell
\binom{\ell}{m} \norm{f^{(m)}}_{L_2(\Omega)}^2 \\
&\le\, 2^s\, \norm{g}_{C^s(\Omega)}^2\,\norm{f}_{H^s(\Omega)}^2\,
\left(\max_{m=0,\dots,s}\sum_{\ell=m}^s \binom{\ell}{m} \right).
\end{split}\]
This proves the bound since
$\sum_{\ell=m}^s \binom{\ell}{m}=\binom{s+1}{m+1}\le 2^{s}$.

(ii) For $\ell=0,1,\dots,s$ we have
$$
(fg)^{(\ell)}(x)=\sum_{m=0}^{\ell}\binom{\ell}{m}f^{(m)}(x)\,g^{(m-\ell)}(x).
$$
Therefore
$$
\|(fg)^{(\ell)}\|_{C(\Omega)}\le
\|f\|_{C^s(\Omega)}\|g\|_{C^s(\Omega)}\,
\sum_{m=0}^{\ell}\binom{\ell}{m}=
\|f\|_{C^s(\Omega)}\|g\|_{C^s(\Omega)}\,2^{\ell},
$$
and
$$
\|fg\|_{C^s(\Omega)}\le 2^s\,\|f\|_{C^s(\Omega)}\|g\|_{C^s(\Omega)}.
$$

\end{proof}


\section{Functions with compact support} \label{sec:compact}

In this section we study the approximation of $I_k^\rho$ for functions from
$H^s_0(\Omega)$ and $C^s_0(\Omega)$ with a bounded $\Omega = [a,b]$
and $\rho =1$.
In particular, we determine the dependence of the optimal error bounds on the
length $|\Omega| = b-a $ of the interval.

For $k\in\R$, we now study the functional
\[
I_{k}(f) \,:=\, \int_\R f(x) \,\eu^{-i\,kx} \, {\rm d} x
\,=\, \int_\Omega f(x) \,\eu^{-i\,kx} \, {\rm d} x
\ \ \ \mbox{for $f\in H^s_0 (\Omega)$ or $f\in C^s_0(\Omega)$}.
\]
First we find upper error bounds
for the initial error and
for a specific algorithm that uses $n$ function values and whose error will
be almost minimal. Then we provide
matching lower bounds.
Similar to \cite[Prop.~3]{NUW13} we prove the following assertion.

\begin{prop}\label{prop:initial1}
The initial error of $I_k$ satisfies
\[
e(0,I_{k}, H^s_0 (\Omega)) \;\le\; \frac{|\Omega|^{1/2}}{\nu_s(k)}
\;\le\; \frac{|\Omega|^{1/2}}{\bar k^s}
\]
and
\[
e(0,I_{k}, C^s_0 (\Omega)) \;\le\; \frac{|\Omega|}{\bar k^s},
\]
with $\nu_s(k) \,:=\,
\sqrt{1+\sum_{\ell=1}^s k^{2\ell}}$ and $\bar k=\max(1,|k|)$.
\qed
\end{prop}
\begin{proof}
Consider the function $e_k(x) = \eu^{i\, k x}$.
Then $I_k(f)=\il f, e_k\ir_{0,\Omega}$ for every $f\in L_1(\Omega)$.
Integration by parts yields
\[
|k|^\ell\, \abs{\il f, e_k\ir_{0,\Omega}}
\,=\, \abs{\il f, e_k^{(\ell)}\ir_{0,\Omega}}
\,=\, \abs{\il f^{(\ell)}, e_k\ir_{0,\Omega}}
\]
for each $\ell=0,1,\dots,s$.
Hence
\[
\nu_s(k)^2 \abs{I_k(f)}^2
\,=\, \sum_{\ell=0}^s \abs{\il f^{(\ell)}, e_k\ir_{0,\Omega}}^2
\,\le\, \sum_{\ell=0}^s \|f^{(\ell)}\|_{L_2(\Omega)}^2 \|e_k\|_{L_2(\Omega)}^2
\,=\, |\Omega|\,\|f\|_{H^s(\Omega)}^2
\]
and
\[
\bar k^{s} \abs{I_k(f)}
\,=\, \max_{\ell=0,\dots,s} \abs{\il f^{(\ell)}, e_k\ir_{0,\Omega}}
\,\le\, \max_{\ell=0,\dots,s} \|f^{(\ell)}\|_{L_\infty(\Omega)}
\|e_k\|_{L_1(\Omega)}
\,=\, |\Omega|\,\|f\|_{C^s(\Omega)}.
\]
Here we used that $\bar k^s=\max_{\ell=0,\dots,s}|k|^\ell$, where by
convention $0^0=1$.
Additionally, note that $\bar k^s\le(1+\sum_{\ell=1}^sk^{2\ell})^{1/2}$.
This completes the proof.
\end{proof}
\vskip 1pc
\begin{rem}
The upper bounds for the initial error can be proved analogously for
the more general Sobolev spaces $W^{s,p}_0(\Omega)$ which are normed
by
\[
\|f\|_{W^{s,p}(\Omega)}\;:=\;
\left(\sum_{\ell=0}^s \|f^{(\ell)}\|_{L_p(\Omega)}^p\right)^{1/p}.
\]
The upper bound would be $\bar k^{-s}\, |\Omega|^{1-1/p}$.
\qed
\end{rem}
\vskip 1pc
We now turn to the definition of an algorithm which uses $n$ function
values and whose error is, as we prove it later, almost minimal.
For $n\ge1$, define
$c_n:=|\Omega|/n$ and
the algorithm
\begin{equation}\label{eq:algorithm1}
A_n^{\Omega}(f) \;=\; c_n\sum_{x\in (c_n \Z)\cap\Omega} f(x) \,\eu^{-i\,kx}
\qquad\text{ for all }\qquad f\in H^s_0 (\Omega).
\end{equation}
Note that $x\in (c_n\Z)\cap\Omega$ means that $x=c_nj\in[a,b]$ for
some integer $j$. Or equivalently,
$$
\frac{j}n\in\left[\frac{a}{b-a},\frac{b}{b-a}\right].
$$
The number of such $j$ is clearly at most $n+1$. In fact, it can be
$n+1$ only if $a/(b-a)=m/n$ for some integer $m$. In this case,
$A_n^{\Omega}(f)$ uses one function value on the left and one on the right
boundary of $\Omega$. Since functions from $f\in H^s_0 (\Omega)$
are zero at these points, they can be omitted from the summation.
Hence, the number of function values used by the algorithm
$A_n^\Omega$ is at most $n$.
We now prove the following error bound for $A_n^\Omega$ for a relatively
large~$n$, whereas the case of small $n$ will be considered later.

\begin{thm} \label{thm:unweighted}
Let $s\in\N$, $k\in\R$ with $\bar k=\max(1,|k|)$,
and $\Omega = [a,b] \subset\R$. The algorithm~$A_n^{\Omega}(f)$
from \eqref{eq:algorithm1}
satisfies
\begin{enumerate}[(i)]
\item for each $f\in H^s_0 (\Omega)$ and $n\ge (2\pi)^{-1}(1+|k|)\, |\Omega| $:
\[
\abs{I_k(f)-A_n^{\Omega}(f)} \;\le\;
\frac{2}{(2\pi)^s}\;\frac{|\Omega|^{1/2}}{(n/
|\Omega|-|k|/2\pi)^s}\;\|f\|_{H^s(\Omega)},
\]
\item
for each $f\in H^s_0 (\Omega)$, $\a\in[1/3,1)$
and $n\ge [(1+\a)/(1-\a)]\,(2\pi)^{-1}\,\bar k\, |\Omega|$:
\[
\abs{I_k(f)-A_n^{\Omega}(f)} \;\le\;
\frac{2}{(2\pi\a)^s}\;\frac{|\Omega|^{1/2}}{\left(n/|\Omega|
+\bar k/2\pi\right)^s}\;\|f\|_{H^s(\Omega)}.
\]
\item
for each $f\in C^s_0 (\Omega)$, $\a\in[1/3,1)$
and $n\ge[(1+\a)/(1-\a)]\,(2\pi)^{-1}\,\bar k\, |\Omega|$:
\[
\abs{I_k(f)-A_n^{\Omega}(f)} \;\le\;
\frac{2}{(\sqrt{2}\,\pi\a)^s}\;\frac{|\Omega|}{\left(n/|\Omega|
+\bar k/2\pi\right)^s}\;\|f\|_{C^s(\Omega)} .
\]
\end{enumerate}
\qed
\end{thm}

\begin{proof}
Let $f\in H_0^s(\Omega)$. Then we can rewrite
\eqref{eq:algorithm1} as
$$
A_n^\Omega(f)\;=\,c_n\sum_{x\in c_n\,\Z}f(x)\eu^{-ikx}.
$$
Using \eqref{eq:poisson} we have
\[
A_n^\Omega(f) \;=\; \sum_{z\in (1/c_n)\Z} [\F f]\Bigl(z+\frac{k}{2\pi}\Bigr)
\;=\; \sum_{z\in \Z} [\F f]\Bigl(\frac{zn}{|\Omega|}+\frac{k}{2\pi}\Bigr),
\]
Noting that $I_k(f)=\F f(k/2\pi)$ we have
\[
\abs{I_k(f)-A_n^{\Omega}(f)}
\;=\; \abs{\sum_{z\in \Z\setminus0} [\F f]\Bigl(\frac{zn}
{|\Omega|}+\frac{k}{2\pi}\Bigr)} .
\]
Define
\begin{equation}\label{eq:proof-nukn}
\nu_{k,n}^s(j) \;:=\;
\left(1+\sum_{\ell=1}^s\abs{2\pi\frac{jn}{|\Omega|}+k}^{2\ell}\right)^{1/2}.
\end{equation}
We bound the error by
\[\begin{split}
\abs{I_k(f)-A_n^{\Omega}(f)}
\;&=\; \abs{\sum_{j\in \Z\setminus0} [\nu_{k,n}^{s}(j)]^{-1}\;
  \nu_{k,n}^s(j)\;[\F f]\Bigl(\frac{jn}{|\Omega|}+\frac{k}{2\pi}\Bigr)} \\
\;&\le\; \left(\sum_{j\in \Z\setminus0} [\nu_{k,n}^{s}(j)]^{-2}\right)^{1/2}
        \left(\sum_{j\in \Z\setminus0} [\nu_{k,n}^{s}(j)]^2\;
   \abs{[\F f]\Bigl(\frac{jn}{|\Omega|}+\frac{k}{2\pi}\Bigr)}^2\right)^{1/2}.
\end{split}
\]
We first bound the second factor.
Integrating by parts yields
$$
\left|[\F f]\Bigl(\frac{jn}{|\Omega|}+\frac{k}{2\pi}\Bigr)\right|^2=
\left(2\pi\,\frac{jn}{|\Omega|}+k\right)^{-2\ell}\,\left|\F [D^\ell f]
 \Bigl(\frac{jn}{|\Omega|}+\frac{k}{2\pi}\Bigr)\right|^2
$$
for all $\ell=0,1,\dots,s$. Summing up with respect to $\ell$, we
obtain
\begin{equation}\label{eq:derivative}
[\nu_{k,n}^{s}(j)]^2\,
\abs{[\F f]\Bigl(\frac{jn}{|\Omega|}+\frac{k}{2\pi}\Bigr)}^2
\;=\; \sum_{\ell=0}^s \abs{\F[D^\ell f]
\Bigl(\frac{jn}{|\Omega|}+\frac{k}{2\pi}\Bigr)}^2 .
\end{equation}
Since $c_n=|\Omega|/n$, from \eqref{eq:derivative} we obtain
\[\begin{split}
\sum_{j\in \Z\setminus0} [\nu_{k,n}^{s}(j)]^2\; &
           \abs{[\F f]\Bigl(\frac{jn}{|\Omega|}+\frac{k}{2\pi}\Bigr)}^2
\;=\; \sum_{\ell=0}^s\sum_{j\in \Z\setminus0}
             \abs{\F[D^\ell f]\Bigl(\frac{j}{c_n}+\frac{k}{2\pi}\Bigr)}^2 \\
\;&=\; \sum_{\ell=0}^s\sum_{j\in \Z\setminus0}
                    \abs{\int_\R D^\ell f(y)\,\eu^{-iky}\,
                      \eu^{-2\pi i \frac{j}{c_n} y}\, {\rm d}y}^2 \\
\;&=\; \sum_{\ell=0}^s\sum_{j\in \Z\setminus0}
                        \abs{c_n \int_\R D^\ell f(c_n x)\,\eu^{-ik c_n x}\,
                                        \eu^{-2\pi i j x}\, {\rm d}x}^2 \\
\;&=\; \sum_{\ell=0}^s\sum_{j\in \Z\setminus0}
 \abs{c_n \sum_{m\in\Z} \int_0^1 D^\ell f(c_n (x+m))\,\eu^{-ik c_n (x+m)}\,
                                        \eu^{-2\pi i j x}\, {\rm d}x}^2. \\
\end{split}\]
Define the function
\[
g_{\ell,n}(x)\;=\; c_n\sum_{m\in\Z} D^\ell f(c_n (x+m))\,\eu^{-ik c_n (x+m)}
\qquad \mbox{for}\ \ x\in[0,1],
\]
and note that for each fixed $x$ the number of non-zero terms in the
last series is
\[
\abs{\{m\in\Z\colon c_n(x+m)\in\Omega\}}
\,=\,\abs{\Z\cap \left(n\frac{\Omega}{|\Omega|}-x\right)} \,\le\, n.
\]
We now show that $g_{\ell,n}\in L_2([0,1])$. Indeed,
$$
|g_{\ell,n}(x)|^2\,\le\, c_n^2\,n\, \sum_{m\in\Z}|D^\ell f(c_n(x+m))|^2,
$$
and
$$
\int_0^1
|g_{\ell,n}(x)|^2\,{\rm d}x\,
\le \,c_n^2\,n\, \int_{\R}|D^\ell f(c_n x)|^2\,{\rm d}x=
c_n\,n\,\int_{\R}|D^\ell f(x)|^2\,{\rm d}x<\infty,
$$
since $f\in H^s_0(\Omega)$ implies that
$D^\ell f\in L_2(\R)$ for all $\ell=0,1,\dots,s$.
Hence, $g_{\ell,n}\in L_2([0,1])$, as claimed.

By Parseval's theorem and the Cauchy Schwarz inequality we have
\[\begin{split}
\sum_{z\in \Z\setminus0} \nu_{k,n}^{s}(j)^2\;
                \abs{[\F f]\Bigl(\frac{jn}{|\Omega|}+\frac{k}{2\pi}\Bigr)}^2
\;&=\; \sum_{\ell=0}^s\sum_{j\in \Z\setminus0}
        \abs{\int_0^1 g_{\ell,n}(x)\,\eu^{-2\pi i j x}\, {\rm d}x}^2 \\
\;&\le\; \sum_{\ell=0}^s \int_0^1 \abs{g_{\ell,n}(x)}^2 \dx \\
\;&\le\; \sum_{\ell=0}^s c_n\, n  \int_\R \abs{D^\ell f(x)}^2 \dx
= |\Omega|\; \|f\|_{H^s(\Omega)}^2  .
\end{split}\]

We now bound the first factor in the estimate of
$|I_kf-A_n^\Omega(f)|$. For this assume that
$n\ge (2\pi)^{-1}(1+|k|)\, |\Omega|$. Then $2\pi n/|\Omega|\ge 1+|k|$.
Since
$\nu_{k,n}^s(j)\ge |2\pi\,j\,n/|\Omega|+k|^s$, we have
\[\begin{split}
\sum_{j\in \Z\setminus0} \nu_{k,n}^{s}(j)^{-2}
\;&\le\; \sum_{j\in \Z\setminus0} \abs{2\pi\frac{jn}{|\Omega|}+k}^{-2s}
\;\le\; 2 \sum_{j=1}^\infty \left(2\pi\frac{jn}{|\Omega|}-\abs{k}\right)^{-2s}\\
\;&=\; 2 \left(2\pi\frac{n}{|\Omega|}-\abs{k}\right)^{-2s}
                + 2\sum_{j=2}^\infty \left(2\pi\frac{jn}{|\Omega|}-\abs{k}\right)^{-2s} \\
\;&\le\; 2 \left(2\pi\frac{n}{|\Omega|}-\abs{k}\right)^{-2s}
                + 2\int_1^\infty \left(2\pi\frac{xn}{|\Omega|}-\abs{k}\right)^{-2s}\dx \\
\;&\le\; 4 \left(2\pi\frac{n}{|\Omega|}-\abs{k}\right)^{-2s}. \\
\end{split}\]
This proves (i).

Let $n\ge[(1+\a)/(1-\a)]\,(2\pi)^{-1}\,\bar k\, |\Omega|$ for
$\alpha\in[1/3,1)$. Then we have
$\frac{n}
{|\Omega|}-\frac{\abs{k}}{2\pi}\ge\a(\frac{n}{|\Omega|}+\frac{\bar{k}}{2\pi})$
and $\bar{k}\,(1+\a)/(1-\a)\ge
1+|k|$. Since now $n\ge (1+|k|)|\Omega|/(2\pi)$, (i) easily yields (ii).
For (iii) we simply
use $\|f\|_{H^s(\Omega)}^2\le(s+1)|\Omega|\,\|f\|_{C^s(\Omega)}^2$
and $\sqrt{s+1}\le2^{s/2}$.
This completes the proof.
\end{proof}
\vskip 1pc
\begin{rem}
It should be noted that we could prove the same upper bounds also for the case
of ``periodic'' functions. This means for
all functions $f\in H^s([a,b])$,
or $C^s([a,b])$,  such that
$f(x)\,\eu^{-ikx}$ is periodic with period $|\Omega|$.
However, we omit it
since this leads to some technicalities and is not needed later. \qed
\end{rem}
\vskip 1pc
Before we state the final result on the $n$th minimal errors,
including matching lower bounds,
we present a modification of the algorithm $A_n^\Omega$ that satisfies good
error bounds also for small $n$.
For small $n$, this algorithm, which we denote by $\A_n^\Omega$,
simply uses no information of the function $f$ and outputs zero.
Although this seems artificial, it is known, at least in special
cases, that for small $n$ the zero algorithm outperforms
$A_n^\Omega$, see~\cite[Thm.~4(ii)]{NUW13}.
More precisely, we define
\begin{equation} \label{eq:A_per}
\A_n^\Omega(f)=\begin{cases}
\ \ \ 0& \ \ \mbox{if} \ \ n<\frac1\pi\,\bar k\, |\Omega|,\\
\ \ \ A_n^\Omega(f)&\ \ \mbox{if}  \ \ n\ge\frac1\pi\,\bar k\, |\Omega|.
\end{cases}
\end{equation}

Theorem~\ref{thm:unweighted} immediately implies the following
error bound on $\A_n^\Omega$.

\begin{cor}\label{cor:unweighted}
For all $n,s\in\N$ and $k\in\R$ with $\bar k=\max(1,|k|)$,
the algorithm $\A_n^\Omega$ satisfies
\[
\abs{I_k(f)-\A_n^{\Omega}(f)} \;\le\;
\frac{2}{2^s}\;\frac{|\Omega|^{1/2}}{\left(n/|\Omega|
+\bar k/2\pi\right)^s}\ \ \|f\|_{H^s(\Omega)}\,
\qquad \text{ for } f\in H^s_0 (\Omega),
\]
\[
\abs{I_k(f)-A_n^{\Omega}(f)} \;\le\;
\frac{2}{2^{s/2}}\;\frac{|\Omega|}{\left(n/|\Omega|
+\bar k/2\pi\right)^s}\;\|f\|_{C^s(\Omega)}
\qquad \text{ for } f\in C^s_0 (\Omega).
\]
\end{cor}

\begin{proof}
Assume first that $n\ge\frac1\pi\,\bar k\, |\Omega|$.
In this case the upper bounds
follow from Theorem~\ref{thm:unweighted} (ii) and (iii)
with $\a=1/3$.
It remains to consider the case $n<\frac1\pi\,\bar k\, |\Omega|$.
Then we have $\bar k>2(n/|\Omega|+\bar k/2\pi)$, and therefore
for $f\in H^s_0 (\Omega)$ with $\|f\|_{H^s(\Omega)}\le1$,
\[
\abs{I_k(f)-\A_n^{\Omega}(f)}
\;\le\; e(0,I_k, H^s_0 (\Omega))
\;\le\; \frac{|\Omega|^{1/2}}{\bar k^s}
\;\le\; \frac{1}{2^s}\;\frac{|\Omega|^{1/2}}{\left(n/|\Omega|+\bar k/2\pi\right)^s},
\]
as claimed.
Again,
we use $\|f\|_{H^s(\Omega)}^2\le 2^s|\Omega|\,\|f\|_{C^s(\Omega)}^2$ for the
second bound.
This completes the proof.
\end{proof}

This enables us to give sharp bounds on the $n$th minimal error.

\begin{thm}\label{thm:unweighted_minimal}  \label{T13}
Let $k\in\R$ with $\bar k=\max(1,|k|)$.
Consider the integration problem $I_k$ defined
for functions from the spaces $H^s_0 (\Omega)$
or $C^s_0 (\Omega)$ with $\Omega = [a,b]$ and
$s\in\N$. Then there exist numbers $c_s\in(0,1/2^{s-1}]$
and $\tilde{c}_s\in(0,1/2^{(s-2)/2}]$ such that for every
       $n \in \N_0$
there are numbers $d_s=d_s(n,k)$ and
$\tilde d_s=\tilde d_s(n,k)$ such that
$$
d_s \in[c_s,1/2^{s-1}]\ \ \  \mbox{and}\ \ \
 \tilde d_s\in[\tilde
c_s,1/2^{(s-2)/2}]
$$
and
\LeftEqNo
\[\tag{i}
e\bigl(n,I_{k}, H^s_0 (\Omega)\bigr)
\;= \;d_s\;
\frac{|\Omega|^{1/2}}{(n/|\Omega|+\bar k/2\pi)^s},
\]
\[\tag{ii}
e\bigl(n,I_{k}, C^s_0 (\Omega)\bigr)
\;=\;\tilde d_s\;
\frac{|\Omega|}{(n/|\Omega|+\bar k/2\pi)^s}, \\[1mm]
\]
Moreover, the lower bounds hold for \emph{all} algorithms that use
at most $n$ function or derivative (up to order $s-1$) values.
\end{thm}
\begin{proof}
The upper bound follows from Corollary~\ref{cor:unweighted}.
The proof of the lower bound is the same as
the proof of Theorem~9 in \cite{NUW13}
with two minor modifications. First, there are now about $\frac1\pi |k||\Omega|$
roots of the oscillatory weight $\cos(kx)$ in $\Omega$ and, secondly,
the fooling function $f$ that is constructed there satisfies
$\|f\|_{H^s(\Omega)}=\Theta(|\Omega|^{1/2})$ or
$\|f\|_{C^s(\Omega)}=\Theta(1)$.
\end{proof}
\vskip 1pc
We stress that the last bounds are sharp with respect to
$n$, $\bar k$ and $|\Omega|$ as well as
with respect to the convergence rate. The only part which is not sharp
involves factors which depend on $s$. However, even the upper bounds on
$d_s$ and $\tilde d_s$ are exponentially small in $s$.

{}From Theorem \ref{T13}
we easily obtain sharp estimates on the information complexities
defined by \eqref{eq:absolute} and \eqref{eq:normalized}.
\begin{cor}\label{cor:compl1}
Let $k\in\R$ with $\bar k=\max(1,|k|)$.
Consider the integration problem $I_k$
defined for
the spaces $H^s_0 (\Omega)$ or $C^s_0(\Omega)$
with $\Omega=[a,b]$ and $s\in\N$. Then for any positive $\e$
\begin{itemize}
\item there exist
$\alpha_s=\alpha_s(\eps,k)\in[c_s,2^{-(s-1)}]$ and
$\beta_s=\beta_s(\eps,k)\in[c_s(2\pi)^{s},1/(c_s2^{s-1})]$
with $c_s$ given in Theorem \ref{T13} such that
\begin{eqnarray*}
n^{\rm abs}\bigl(\eps,I_{k}, H^s_0 (\Omega)\bigr)\,&=&\,
\left\lceil |\Omega|\,
\left( \left(\frac{\alpha_s|\Omega|^{1/2}}
{\eps}\right)^{1/s} - \frac{\bar k}{2\pi}
 \right)_+\right\rceil,\\
n^{\rm nor}\bigl(\eps,I_{k}, H^s_0 (\Omega)\bigr)\,&=&\,
\left\lceil|\Omega|\,\frac{\bar k}{2\pi}
\left(\left(\frac{\beta_s}{\eps}\right)^{1/s} - 1
 \right)_+\right\rceil,
\end{eqnarray*}
\item there exist
$\tilde\alpha_s=\tilde\alpha_s(\eps,k)\in[\tilde c_s,2^{-(s-2)/2}]$ and
$\tilde \beta_s=
\tilde\beta_s(\eps,k)\in[\tilde c_s(2\pi)^s,1/(\tilde c_s2^{(s-2)/2})]$
with $\tilde c_s$ given in Theorem~\ref{T13} such that
\begin{eqnarray*}
n^{\rm abs}\bigl(\eps,I_{k}, C^s_0 (\Omega)\bigr)\,&=&\,
\left\lceil |\Omega|\,
\left(\left(\frac{\tilde\alpha_s|\Omega|}
{\eps}\right)^{1/s} - \frac{\bar k}{2\pi}
 \right)_+\right\rceil,\\
n^{\rm nor}\bigl(\eps,I_{k}, C^s_0(\Omega)\bigr)\,&=&\,
\left\lceil |\Omega|\,\frac{\bar k}{2\pi}\,
\left(\left(\frac{\tilde\beta_s}{\eps}\right)^{1/s} - 1
\right)_+\right\rceil.
\end{eqnarray*}
\end{itemize} \qed
\end{cor}

\begin{proof}
The results for
the absolute error criterion, i.e. the bounds on $n^{\rm abs}$, are
obvious from Theorem~\ref{T13}.
In view of \eqref{eq:normalized} the information complexity for the normalized error
criterion is given by
$n^{\rm nor}\bigl(\eps,I_{k}, F\bigr)=n^{\rm abs}\bigl(\eps\cdot e(0,I_{k}, F),I_{k}, F\bigr)$
for $F\in\{H^s_0 (\Omega),C^s_0(\Omega)\}$.
From Theorem~\ref{T13} (for $n=0$) and Proposition~\ref{prop:initial1} we know that
$e(0,I_{k}, H^s_0 (\Omega))\in\left[c_s |\Omega|^{1/2}(\bar k/2\pi)^{-s}, |\Omega|^{1/2}\bar k^{-s} \right]$
and $e(0,I_{k}, C^s_0 (\Omega))\in[\tilde c_s |\Omega|(\bar k/2\pi)^{-s},
 |\Omega|\bar k^{-s}]$.
Putting this in the bounds on $n^{\rm abs}$ this shows
$(2\pi)^s \alpha_s \le \beta_s \le c_s^{-1} \alpha_s$
and $(2\pi)^s \tilde\alpha_s \le \tilde\beta_s \le c_s^{-1} \tilde\alpha_s$
and proves the claim.
\end{proof}

The formulas in Corollary~\ref{cor:compl1} can be simplified when $\e$
goes to zero. Then we have
\begin{eqnarray*}
n^{\rm abs}\bigl(\eps,I_{k}, H^s_0 (\Omega)\bigr)\,&=&\,\Theta\left(
\frac{|\Omega|^{1+1/(2s)}}{\e^{1/s}}\right),\\
n^{\rm nor}\bigl(\eps,I_{k}, H^s_0 (\Omega)\bigr)\,&=&\,\Theta\left(
\frac{|\Omega|\,\bar k}{\e^{1/s}}\right),\\
n^{\rm abs}\bigl(\eps,I_{k}, C^s_0 (\Omega)\bigr)\,&=&\,\Theta\left(
\frac{|\Omega|^{1+1/s}}{\e^{1/s}}\right),\\
n^{\rm nor}\bigl(\eps,I_{k}, C^s_0(\Omega)\bigr)\,&=&\,\Theta\left(
\frac{|\Omega|\,\bar k}{\e^{1/s}}\right),
\end{eqnarray*}
where the factors in the $\Theta$ notations depend only on $s$.

We stress that we now have sharp dependence on $\e,|\Omega|$ and $\bar
k$. In all cases the dependence on $\e$ is through $\e^{-1/s}$,
whereas the dependence on $|\Omega|$ and $\bar k$ varies and is
different for the spaces $H^s_0(\Omega)$ and $C^s_0(\Omega)$ as well
as it depends on  the error criterion.
For the absolute error criterion there is
no asymptotic dependence on $\bar k$, however, for large $\bar k$ we
have to wait longer to see this asymptotic dependence.
For the normalized error criterion,
the information
complexity of the integration problem $I_k$ is roughly the same for
$H_0^s(\Omega)$ and $C_0^s(\Omega)$ and the dependence on $|\Omega|$
and $\bar k$ is through $|\Omega|\,\bar k$. In this case, large
$\bar k$ hurts.

Observe that the dependence on the size of
$|\Omega|$ with $|\Omega|>1$
 is more
severe for the absolute than for the normalized error
criterion, however, for large $s$ this difference disappears.
For small $|\Omega|<1$, the opposite holds and
the absolute error criterion is easier than the normalized error
criterion.


\section{Functions on the real line}\label{sec:line}

We now  consider the approximation of integrals of the form
$$
I_k^\rho (f) = \int_{\R} f(x) \,\eu^{-i\,kx} \rho(x)  \, {\rm d} x
\ \ \ \mbox{for}\ \ f\in H^s(\R)\ \ \mbox{or}\ \ f\in C^s(\R)
$$
with a sufficiently smooth density function $\rho$. The primal example
of such $\rho$ is the normal density
\[
\rho(x) =  \frac{1}{\sqrt{2 \pi}} \exp( -x^2/2) \ \ \ \mbox{for}\ \ x\in\R.
\]
We establish conditions on $\rho$
such that the optimal error bounds are
of the order $(n+ \bar k )^{-s}$, just as in the case
for a bounded interval with functions of compact support,
see Theorem~\ref{T13}.

We need the notion of a smooth partition of unity.
We call a family
$\{g_m\}_{m\in\Z}$ of functions a \emph{smooth partition of unity}
if $g_m\in C^\infty_0(\R)$ for all $m\in\Z$ and
\[
\sum_{m\in\Z}g_m(x)=1 \qquad \text{ for all } x\in\R.
\]
In this section we use a partition of unity with a specific structure.
Namely, we choose a (fixed) nonnegative
function $g\in C^\infty(\R)$ such that
$$
{\rm supp}(g)=[-1,1]\ \ \ \mbox{and}\ \ \
g(x)+g(x-1)=1\ \ \mbox{for}\ \ x\in[0,1],
$$
and define the functions
$$
g_m(x)=g(x-m)\ \ \ \mbox{for}\ \  m\in\Z\ \ \mbox{and} \ \ x\in\R.
$$
Such functions $g$ obviously exist; consider for example the up-function
defined in~\cite{Rv90}.
This function is defined by
$$
\frac{1}{2\pi} \int_\R e^{itx} \prod_{k=1}^\infty \frac{\sin t2^{-k}}{t2^{-k}}
{\rm d} t,
$$
it is a solution of the equation
$$
y' (x) = 2y (2x+1) -2y (2x-1)
$$
with compact support.


Based on such a partition of unity $\{g_m\}_{m\in\Z}$,
we define the partition of
the density $\rho$ by $\{\rho_m\}_{m\in\Z}$ with $\rho_m=g_m\cdot\rho$.
Clearly, ${\rm supp}(\rho_m)=[m-1, m+1]$, $\sum_{m\in\Z}\rho_m=\rho$,
and therefore
\begin{equation}\label{split}
I_k^\rho(f)=\sum_{m\in\Z}I_k^{\rho_m}(f)=\sum_{m\in\Z}I_k(f \rho_m).
\end{equation}
Note that $f \rho_m$ is now a function with support $\Omega_m:=[m-1, m+1]$
for each $m\in\Z$. If $f\rho_m$ belongs to $H^s_0(\Omega_m)$ or to
$C^s_0(\Omega_m)$ then we can approximate
$I_k(f \rho_m)$ by an (almost) optimal algorithm from
the last section, see~\eqref{eq:A_per} and Corollary~\ref{cor:unweighted}.
If this holds for all $m\in\Z$, we can apply the algorithm
from~\eqref{eq:A_per} to
each piece $f \rho_m$ and obtain a good approximation of $I_k^\rho(f)$.
We will prove that this is indeed the case under some smoothness
assumptions on $\rho$.

The algorithm for the approximation of $I_k^\rho(f)$ we are going
to analyze is based on a suitable distribution of the $n$ function
evaluations over the real line.
For this, let
$$
\prob=\{p_m\}_{m\in\Z}
$$
be a family of real numbers, to be specified in a moment,  such that
\begin{equation}\label{propp}
p_m\in[0,1]\ \ \  \mbox{for all $m$ \ \ and \ \ $\sum_{m\in\Z}p_m=1$}.
\end{equation}
For the space $H^s(\R)$ we assume that $\rho\in C^s(\R)$ and
that the sequence $\{\|\rho\|_{C^s(\Omega_m)}\}_{m\in\Z}$ belongs to
  $\ell_{1/(s+1/2)}$, i.e.,
\begin{equation}\label{pp1}
\sum_{m\in\Z}\|\rho\|_{C^s(\Omega_m)}^{1/(s+1/2)}<\infty,
\end{equation}
whereas for the space $C^s(\R)$ we assume that
$\rho\in H^s(\R)$ and $\{\|\rho\|_{H^s(\Omega_m)}\}_{m\in\Z}\in
\ell_{1/(s+1)}$,
i.e.,
\begin{equation}\label{pp2}
\sum_{m\in\Z}\|\rho\|_{H^s(\Omega_m)}^{1/(s+1)}<\infty.
\end{equation}
It is easy to verify that these assumptions hold if $\rho$ is $s$
times continuously differentiable and its derivatives decay
exponentially fast to zero if its argument goes to infinity.

In particular, this is true for the normal density
$\rho(x)=1/\sqrt{2\pi\sigma^2}\,\exp(-x^2/(2\sigma^2))$
with the standard deviation $\sigma >0$.
This problem seems to be of some physical importance.
In this case,
we can estimate the norms  $\|\rho\|_{C^s(\Omega_m)}$ for $m \in \Z$
by Cramer's bound which states that
\[
\|\rho\|_{C^s(\Omega_m)}
\,\le\, (2\pi)^{-1/4}\,\sigma^{-1}\, \sqrt{s!}\, e^{-(\bar{m}-1)^2/(4\sigma^2)},
\]
with $\bar m=\max(1,|m|)$, see e.g.~\cite[p.~324]{Sa77}.
Clearly, the sequence $\bigl\{\|\rho\|_{C^s(\Omega_m)}\bigr\}_{m\in\Z}$
(and hence the sequence $\bigl\{\|\rho\|_{H^s(\Omega_m)}\bigr\}_{m\in\Z}$) is in
every $\ell_p$, $p>0$, due to its exponential decay.
Therefore~\eqref{pp1} and \eqref{pp2} hold.

Note that due to Lemma~\ref{lem:norm} (ii),
\eqref{pp1} implies that
$$
\rho_{C^s}:=\sum_{m\in\Z}\|\rho_m\|_{C^s(\Omega_m)}^{1/(s+1/2)}\le
2^{s/(s+1/2)}\,\|g\|_{C^s(\R)}^{1/(s+1/2)}\,\sum_{m\in \Z}
\|\rho\|_{C^s(\Omega_m)}^{1/(s+1/2)}<\infty,
$$
whereas due to Lemma~\ref{lem:norm} (i), \eqref{pp2} implies that
$$
\rho_{H^s}:=\sum_{m\in\Z}\|\rho_m\|_{H^s(\Omega_m)}^{1/(s+1)}\le
2^{s/(s+1)}\,\|g\|_{C^s(\R)}^{1/(s+1)}\,\sum_{m\in \Z}
\|\rho\|_{H^s(\Omega_m)}^{1/(s+1)}<\infty.
$$

Then
\begin{eqnarray*}
p_m&=&\frac{\|\rho_m\|_{C^s(\Omega_m)}^{1/(s+1/2)}}{\rho_{C^s}}
\ \ \ \ \
\mbox{for}\
\ H^s(\R),\\
p_m&=&\frac{\|\rho_m\|_{H^s(\Omega_m)}^{1/(s+1)}}{\rho_{H^s}}
\ \ \ \ \ \ \ \mbox{for}\
\ C^s(\R)
\end{eqnarray*}
are well defined and satisfy \eqref{propp}.

We are ready to define the algorithm for approximating $I_k^\rho$ by
\begin{equation}\label{eq:A}
\mathcal{A}_{n,\prob}(f) = \sum_{m\in\Z}\,
\mathcal{A}_{n_m}^{\Omega_m}(f\cdot\rho_m)\qquad
{\rm with}\;\; n_m\,:=\, \lfloor p_m\, n\rfloor,
\end{equation}
where $n \in\N_0$, $f\in H^s(\R)$ or $f\in C^s(\R)$,  and
$\mathcal{A}_{n_m}^{\Omega_m}$ is given by~\eqref{eq:A_per}.
Note that $f\in H^s(\R)$ implies that $f\rho_m\in H^s(\Omega_m)$
and $f\in C^s(\R)$ implies that $f\rho_m\in C^s(\Omega_m)$. Hence
$\mathcal{A}_{n_m}^{\Omega_m}(f\cdot\rho_m)$ is well defined.
Note also that only finitely many $n_m$ are nonzero. Therefore almost
all $n_m=0$, and since $\mathcal{A}_0^{\Omega_m}=0$ the series in \eqref{eq:A}
has only finitely many nonzero terms. Hence, $\mathcal{A}_{n,\prob}$
is well defined. We now estimate the error of $\mathcal{A}_{n,\prob}$.
\medskip
\begin{thm}\label{algrealline}
Assume that \eqref{pp1} holds if we consider the space $H^s(\R)$,
and \eqref{pp2} if we consider the space $C^s(\R)$.
Let $\mathcal{A}_{n,\prob}$ be given by \eqref{eq:A} for $n\in \N_0$,
and let $\bar k=\max(1,|k|)$.  Then
\begin{eqnarray*}
|I_k^\rho(f)-\mathcal{A}_{n,\prob}(f)|&\le&
 \frac{4(2\pi)^s\rho_{C^s}^{s+1/2}}{(n+\bar k)^s}\ \|f\|_{H^s(\R)}\ \
 \  \mbox{for all}\ \ f\in H^s(\R),\\
|I_k^\rho(f)-\mathcal{A}_{n,\prob}(f)|&\le&
 \frac{2^{3/2}(2\pi)^s\rho_{H^s}^{s+1}}{(n+\bar k)^s}\ \|f\|_{C^s(\R)}\ \
 \  \mbox{for all}\ \ f\in C^s(\R).
\end{eqnarray*}
\qed
\end{thm}
\begin{proof}
{}From \eqref{split}
the error of the algorithm $\A_{n,\prob}$ can be bounded by
\[\begin{split}
|I_k^\rho(f)- \A_{n,\prob}(f)|\;&\le\;\sum_{m\in\Z}\,|I_k(f\cdot\rho_m)
- \A_{n_m}^{\Omega_m}(f\cdot\rho_m)|.
\end{split}\]
We need to express the right hand side of this inequality in terms of $p_m n$
instead of $n_m$.
For $n_m = \lfloor p_m\,n \rfloor < 2\bar k /\pi$ we have
$\A_{n_m}^{\Omega_m}=0$, see~\eqref{eq:A_per}. Note that now
$p_m n + \bar k < 2 \bar k/\pi +1 +\bar k \le \bar k(2/\pi+2)\le 3 \bar k$.
Since $|\Omega_m|=2$, we obtain from Proposition~\ref{prop:initial1}
that
\begin{equation}\label{123}
\frac{|I_k(f\cdot\rho_m)- \A_{n_m}^{\Omega_m}
(f\cdot\rho_m)|}{\norm{f\cdot\rho_m}_{H^s(\Omega_m)}}
\,\le\, \frac{\sqrt{2}}{\bar k^s}
\,\le\, \frac{2^{1/2}3^s}{(p_m n+\bar k)^s}.
\end{equation}
For $n_m = \lfloor p_m\,n \rfloor \ge 2\bar k /\pi$, we cannot have
$n_m=0$. Therefore $n_m\ge1$. We also have
$n_m\ge p_m n/\pi$. Indeed, it is true if $p_mn\le \pi$ and for
$p_mn>\pi$ we have $n_m\ge p_mn-1\ge p_mn/\pi$.
Hence
$n_m +\bar k/\pi \,\ge\, \frac1{\pi}(p_m n+\bar k)$.
This and Corollary~\ref{cor:unweighted} yield
\begin{equation}\label{124}
\frac{|I_k(f\cdot\rho_m)- \A_{n_m}^{\Omega_m}
(f\cdot\rho_m)|}{\norm{f\cdot\rho_m}_{H^s(\Omega_m)}}
\,\le\, \frac{2^{3/2} \pi^s}{(p_m n+\bar k)^s}.
\end{equation}
Due to \eqref{123}, the last inequality \eqref{124} holds
for all $n$. Since $(p_mn+\bar k)^s\ge p_m^s(n+\bar k)^s$, we
have
\begin{equation}   \label{both}
\begin{split}
|I_k^\rho(f)- \A_{n,\prob}(f)|
\;&\le\;\sum_{m\in\Z}\,\frac{2^{3/2} \pi^s}{\left(p_m n+\bar k\right)^s}
\;\norm{f\cdot\rho_m}_{H^s(\Omega_m)} \\
\;&\le\; \frac{2^{3/2} \pi^s}{\left(n+\bar k\right)^s}\,
\sum_{m\in\Z}\,\frac{\norm{f\cdot\rho_m}_{H^s(\Omega_m)}}{p_m^s}.
\end{split}
\end{equation}

It remains to bound the last sum.
We first prove the result for $f\in H^s(\R)$.
In this case,
$p_m=\rho_{C^s}^{-1}\,\|\rho_m\|_{C^s(\Omega_m)}^{1/(s+1/2)}$.
Clearly,
\[
\sum_{m\in\Z} \norm{f}_{H^s(\Omega_m)}^2 \,=\, 2\, \norm{f}_{H^s(\R)}^2.
\]
By Lemma~\ref{lem:norm}(i), we obtain
\[\begin{split}
\sum_{m\in\Z}\,\frac{\norm{f\cdot\rho_m}_{H^s(\Omega_m)}}{p_m^s}
\;&\le\; 2^s \sum_{m\in\Z}\,\frac{1}{p_m^s}\;
\norm{\rho_m}_{C^s(\Omega_m)}\;\norm{f}_{H^s(\Omega_m)}\\
\;&\le\; 2^s\left(\sum_{m\in\Z}\,
\frac{\norm{\rho_m}_{C^s(\Omega_m)}^2}{p_m^{2s}} \right)^{1/2}
        \left(\sum_{m\in\Z}\,\norm{f}_{H^s(\Omega_m)}^2\right)^{1/2}\\
\;&=\; 2^{s+1/2} \norm{f}_{H^s(\R)}\,
\left(\sum_{m\in\Z}\,\frac{\norm{\rho_m}_{C^s(\Omega_m)}^2}{p_m^{2s}} \right)^{1/2}\\
\;&=\; 2^{s+1/2}\; \rho_{C^s}^s \norm{f}_{H^s}\,
\left(\sum_{m\in\Z}\,\norm{\rho_m}_{C^s(\Omega_m)}^{1/(s+1/2)} \right)^{1/2}.
\end{split}\]
With \eqref{both}
this leads to
\[
|I_k^\rho(f)- \A_{n,\prob}(f)|
\;\le\; \frac{\wt{C}_{s,\rho}\,\norm{f}_{H^s(\R)}}{\left(n+\bar k\right)^s},
\]
where $\wt C_{s,\rho} = 4 (2\pi)^{s}\, \rho_{C^s}^{s+1/2}$,
and proves the first estimate.

For $f\in C^s(\R)$, we have
$p_m=\rho_{H^s}^{-1}\,\|\rho_m\|_{H^s(\Omega_m)}^{1/(s+1)}$ and
Lemma~\ref{lem:norm} (i) yields
\[\begin{split}
\sum_{m\in\Z}\,\frac{\norm{f\cdot\rho_m}_{H^s(\Omega_m)}}{p_m^s}
\;&\le\; 2^s \sum_{m\in\Z}\,\frac{1}{p_m^s}\;\norm{f}_{C^s(\Omega_m)}\,
\norm{\rho_m}_{H^s(\Omega_m)}\\
\;&\le\; 2^s\,\norm{f}_{C^s(\R)}\,
\sum_{m\in\Z}\,\frac{\norm{\rho_m}_{H^s(\Omega_m)}}{p_m^s}
\;=\; (2 \rho_{H^s})^{s}\,\norm{f}_{C^s}\,
\sum_{m\in\Z}\,\norm{\rho_m}_{H^s}^{1/(s+1)}\\
\;&=\; 2^s\, {\rho_{H^s}}^{s+1}\,\norm{f}_{C^s(\R)}.
\end{split}\]
Hence
\[
|I_k^\rho(f)- \A_{n,\prob}(f)|
\;\le\; \frac{C_{s,\rho_{H^s}}\,\norm{f}_{C^s(\R)}}{\left(n+\bar k\right)^s}
\]
with $C_{s,\rho}= 2^{3/2}(2\pi)^{s}\,{\rho_{H^s}^{s+1}}$.
This proves the second estimate and completes the proof.
\end{proof}

We are ready to present sharp bounds on the $n$th minimal errors.

\begin{thm} \label{thm:rho}
Consider the integration problem $I^\rho_k$ defined over
the spaces $H^s(\R)$ or $C^s(\R)$ with $s\in\N$ and a nonzero
density $\rho$.
Then for every $n \in \N_0$ and $k \in \R$ we have
\begin{eqnarray*}
 e\big(n, I_k^\rho, H^s(\R)\big)&=&\Theta\left((n+\bar k)^{-s}\right)
\qquad \text{ if }\ \
\left\{\| \rho \|_{C^s(\Omega_m)}\right\}_{m\in\Z}\in\ell_{1/(s+1/2)},\\
 e\big(n, I_k^\rho, C^s(\R)\big)&=&\Theta\left((n+\bar k)^{-s}\right)
\qquad \text{ if }\ \
\left\{\|\rho\|_{H^s(\Omega_m)}\right\}_{m\in\Z}\in\ell_{1/(s+1)},
\end{eqnarray*}
where the factors in the $\Theta$ notation depend only on $s$ and
$\rho$. As always, $\bar k=\max(1,|k|)$.
\qed
\end{thm}
\begin{proof}
The proof of lower bounds can be done
as in the proof
of Theorem~9 in \cite{NUW13}.
We only use the fact that $\rho$ is continuous
and different than zero and conclude that
$$
e(n, I^\rho_k, F) \ge c_{\rho,s} (n+ \bar k )^{-s}\ \ \
\mbox{for}\ \  F=H^s(\R) \ \ \mbox{and}\ \ F=C^s(\R).
$$
Note that for $n=0$ we have a lower bound on the initial error.
The upper bounds are attained by the algorithm $\mathcal{A}_{n,\prob}$
whose error is bounded in Theorem~\ref{algrealline}.
\end{proof}
The assumptions on $\rho$ in Theorem~\ref{thm:rho} for
$H^s(\R)$ and $C^s(\R)$ differ in the
conditions on the decay of $\rho$ at infinity.
One reason for this difference is that
the space $C^s(\R)$ does not guarantee any integrability property.
We did not try to find optimal conditions on~$\rho$.

The essence of Theorem~\ref{thm:rho} is that
if $\rho$ is smooth enough and decays fast enough at infinity, we see
that the $n$th minimal errors for the integration problem
$I_k^\rho$ for the spaces $H^s(\R)$ and $C^s(\R)$
are of the same order and may be different only in the
dependence on $s$ in the factors in the $\Theta$ notation.

We now rewrite Theorem~\ref{thm:rho} in terms of the information complexities
similarly as it was done in Corollary~\ref{cor:compl1}.
\begin{cor}\label{cor:compl11}
Consider the integration problem $I_k$ for $k\in\R$ defined for
the space $F\in \{H^s(\R),C^s(\R)\}$
with $s\in\N$. Then
\begin{eqnarray*}
n^{\rm abs}\bigl(\eps,I_{k}, F\bigr)\,&=&\,
\left\lceil\left(
\Theta\left(\left(\frac1{\e}\right)^{1/s}\right)-\bar k\right)_+
\right\rceil,\\
n^{\rm nor}\bigl(\eps,I_{k}, F\bigr)\,&=&\
\left\lceil\,\bar k\left(
\Theta\left(
\left(\frac1{\e}\right)^{1/s}\right)-1\right)_+
\right\rceil,
\end{eqnarray*}
where the factors in the $\Theta$ notations depend only on $s$
and $\rho$.
\qed
\end{cor}
Hence, modulo the dependence on $s$ in the factors in $\Theta$ notations,
the information complexities for
the spaces $H^s(\R)$ and $C^s(\R)$ are the same. The parameter
$\bar k$ helps for the absolute error criterion  although it does not
change the asymptotic behaviour of the information complexity
when $\e$ goes to zero. The parameter $\bar k$ plays a different role
for the normalized error criterion since the information complexity is
now proportional to $\bar k$.
\noindent
\subsubsection*{Acknowledgement}
S. Zhang is sincerely thankful to Theoretical Numerical Analysis Group
at Universit\"at Jena since the
work was partly done when he was visiting this group.

E. Novak is supported by the DFG-Priority Program 1324. M. Ullrich is supported by the DFG-GRK 1523.
H. Wo\'zniakowski is
supported by the National Science Centre, Poland, based on the
decision DEC-21013/09/B/ST1/04275.
S. Zhang is
supported by the NNSF of China
(11301002), Anhui Provincial Natural Science Foundation (1408085QF107)
and Talents Youth Fund of Anhui Province Universities
(2013SQRL006ZD).

\end{document}